\documentclass[10pt,a4paper]{article}
\usepackage[utf8x]{inputenc}
\usepackage{amsmath, amsthm}
\usepackage{amsfonts}
\usepackage{amssymb,  mathbbol}
\usepackage{graphicx, caption}
\usepackage{microtype, enumerate}
\usepackage[table]{xcolor}
\usepackage{parskip}
\usepackage[margin=3cm]{geometry}
\usepackage{pdfpages}
\usepackage{authblk}

\makeatletter
\def\thm@space@setup{%
  \thm@preskip=\parskip \thm@postskip=0pt
}
\makeatother
 \newtheorem{theorem}{Theorem}[section]
 \newtheorem{lemma}[theorem]{Lemma}
 \newtheorem{remark}[theorem]{Remark}
 \newtheorem{prop}[theorem]{Proposition}
 \newtheorem{conjecture}[theorem]{Conjecture}
\theoremstyle{definition}
 \newtheorem{definition}[theorem]{Definition}
 \newtheorem{questions}[theorem]{Questions}
\newtheorem{exx}[theorem]{Example} 
\newenvironment{ex}
  {\pushQED{\qed}\exx}
  {\popQED\endexx}

\newcommand{\st}{\operatorname{st}}
\newcommand{\HF}{\operatorname{HF}}
\newcommand{\init}{\operatorname{in}}
\newcommand{\Lex}{\operatorname{Lex}}
\newcommand{\RevLex}{\operatorname{RevLex}}
\newcommand{\spn}{\operatorname{span}}

\newcommand{\K}{\mathbb{k}}

\begin{document}
\title{Subalgebras generated in degree two with minimal Hilbert function}

\author{Lisa Nicklasson\thanks{During the preparation of this work the author was partially supported by INdAM.
}}
\affil{Stockholm University \\ \normalsize \texttt{lisan@math.su.se} }
\date{}

\maketitle

\begin{abstract}
 \noindent What can be said about the subalgebras of the polynomial ring, with minimal or maximal Hilbert function? This question was discussed in a recent paper by M. Boij and A. Conca. In this paper we study the subalgebras generated in degree two with minimal Hilbert function. The problem to determine the generators of these algebras transfers into a combinatorial problem on counting maximal north-east lattice paths inside a shifted Ferrers diagram. We conjecture that the subalgebras generated in degree two with minimal Hilbert function are generated by an initial Lex or RevLex segment.
\end{abstract}

\section{Introduction}
In a recent paper by Boij and Conca \cite{Boij-Conca} the Hilbert function of a subalgebra of the polynomial ring is studied. They ask what can be said about the upper and lower bounds for the Hilbert function, in terms of the number of variables, the number of generators of the subalgebra, and the degree of the generators. This question is inspired by the Fröberg conjecture \cite{Frob} on the minimal Hilbert series of the quotient of a the polynomial ring with a homogeneous ideal. For a review of the Fröberg conjecture and related problems, see \cite{Frob-Lund}. In this note we will focus of subalgebras generated in degree two, with minimal Hilbert function. We conjecture that these algebras are always given by a Lex or RevLex segment, see Conjecture \ref{conj}. This conjecture is proved for three large classes of algebras in Theorem \ref{thm:main}. For the first class, the proof is by a computer computation, and for the other two by inductive arguments, using the first class as the base. 

Let $\K$ be a field, and let $R=\K[x_1, \ldots x_n]$ be the standard graded polynomial ring in $n$ variables. Let $R_d$ denote the $\K$-space of homogeneous polynomials of degree $d$ in $R$. For a linearly independent subset $W \subseteq R_d$, let $\K[W]\subseteq R$ be the subring of $R$ generated by the elements in $W$. Define the Hilbert function of such an algebra $\K[W]$ as $\HF(\K[W],i)=\dim_\K(\spn W^i)$. Given positive integers $u$ and $i$, how should we choose $W$ so that $|W|=u$ and $ \HF(\K[W],i)$ takes the smallest possible value? Proposition 3.3 in \cite{Boij-Conca} states that we should choose $W$ as a strongly stable set of monomials. 

\begin{definition}
 A set $W$ of monomials in $R_d$ is called \emph{strongly stable} if $m \in W$ and $x_i | m$ implies $(x_j/x_i) m \in W$ for all $j<i$. 
 
 We use the notation $\st(m_1, \ldots, m_s)$ for the smallest strongly stable set containing the monomials $m_1, \ldots, m_s$, and we say that $m_1, \ldots, m_s$ are \emph{strongly stable generators} of this set. 
\end{definition}

Let $L(n,d,u,i)$ denote the minimal value of $\HF(\K[W],i)$ among all strongly stable subsets $W \subseteq R_d$ of size $u$. The following three questions \cite[Questions 3.6]{Boij-Conca} are asked, for fixed parameters $d$, $n$, and $u$.
\begin{enumerate}
\item Is there a $W$ such that $\HF(\K[W],i)=L(n,d,u,i)$ for all $i$? 
\item Given $i$, can one characterize combinatorially the strongly stable set(s) $W$ such that $\HF(\K[W],i)$ $=L(n,d,u,i)$?
\item Suppose we have $W$ such that $ \HF(\K[W],2)=L(n,d,u,2)$. Does it follow that $\HF(\K[W],i)=L(n,d,u,i)$ for all $i$?
\end{enumerate}
 
 We will see in Example \ref{ex:persist_counterex} that the answer to the questions 1 and 3 is ``No''. Since there is not one generating set $W$ that minimizes $\HF(\K[W],i)$ for all $i$, it is not obvious what the meaning of ``minimal Hilbert function'' should be. For given parameters $n$, $u$, and $d$, there is always a finite number of strongly stable sets to consider. For each set we know that the Hilbert function is given by the Hilbert polynomial, for $i$ large enough. It follows that there will be an algebra with asymptotically minimal Hilbert function. From now on, we say that $\K[W]$ has \emph{minimal Hilbert function} if it is minimal in the asymptotic sense. That is, $\K[W]$ has minimal Hilbert function if there is a number $N$ such that $\HF(\K[W],i)=L(n,d,u,i)$ for all $i>N$. 
 Assuming this definition of minimal Hilbert function, it makes sense to specialize question 2 as follows.
 
 \emph{How can the strongly stable set(s) $W$ such that $\K[W]$ has minimal Hilbert function be characterized?}
 
 The aim of this paper is to study this question in the smallest non-trivial case w.\,r.\,t.\ the parameter $d$. Hence we fix $d=2$ from now on, and focus on strongly stable sets of monomials of degree two. An advantage with this restriction is the connection to combinatorics, as we will see in Section \ref{sec:mult}. The case $d \ge 3$ is discussed in Section \ref{sec:last}.
 
 To minimize the Hilbert function we firstly want to minimize the degree of the Hilbert polynomial. If there are exactly $n$ variables that occurs in the monomials in $W$, the degree of the Hilbert polynomial is $n-1$. Hence, for a given $u$, we want to choose a strongly stable set of $u$ monomials in as few variables as possible. Secondly, we want to minimize the leading coefficient of the Hilbert polynomial. Recall that, if the Hilbert polynomial is of degree $n-1$, the leading coefficient multiplied by $(n-1)!$ is the \emph{multiplicity} of the algebra, which we will denote $e(\K[W])$.

 \section{The multiplicity of subalgebras generated by a strongly stable set of degree two}\label{sec:mult}
 Strongly stable sets of quadratic monomials are also considered as bases of \emph{specialized Ferrers ideals}, which are studied in e.\,g.\ \cite{Corso-Nagel,CNPY}. These sets can be illustrated by a diagram, as in Figure \ref{fig:first_example}. The box in row $i$ and column $j$ corresponds to the monomial $x_ix_j$. Since $x_ix_j=x_jx_i$ we only need to consider boxes on and above the diagonal in the diagram. That the set is strongly stable means precisely that if the box on position $(i,j)$ is included in the diagram, so is everything above and to the left of $(i,j)$.  
 
\begin{figure}[ht]
\begin{minipage}[t]{0.6\textwidth}
\includegraphics[scale=0.8]{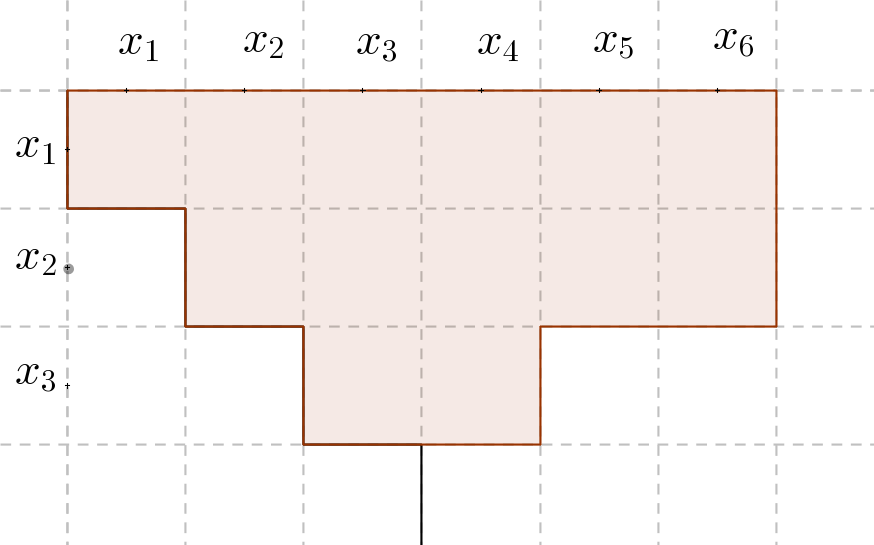}
\captionsetup{singlelinecheck=off}
\caption[.]{\vspace{-0.8cm}\begin{align*}
&\st(x_2x_6, x_3x_4) = \\
&\{x_1^2, x_1x_2, x_1x_3, x_1x_4, x_1x_5, x_1x_6, \\
& \quad x_2^2, x_2x_3, x_2x_4, x_2x_5, x_2x_6, x_3^2, x_3x_4\}
\end{align*}}\label{fig:first_example}
\end{minipage}\begin{minipage}[t]{0.4\textwidth}
\includegraphics[scale=0.7]{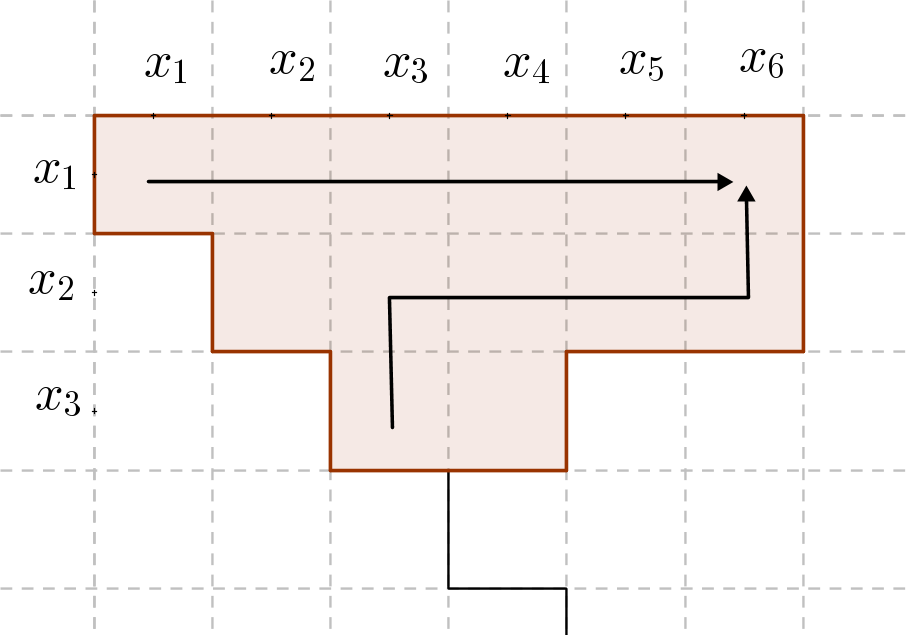}
\caption{Maximal NE-paths in a diagram.}\label{fig:paths}
\end{minipage}
\end{figure}
 
Define an NE-path to be a lattice path in the diagram that can only go up or right (north or east). We say that an NE-path is \emph{maximal} if it is of maximal length, which implies that it starts in $x_i^2$ (on the diagonal) for some $i$, and goes to $x_1x_n$ (the upper right corner). An example of two such paths can be found in the Figure \ref{fig:paths}. It is proved in \cite[Theorem 4.2]{CNPY} that if $W$ is a strongly stable set of degree two monomials, then $\K[W]$ is isomorphic to a certain determinantial ring, which has a defining ideal with a quadratic Gröbner basis. It then follows from \cite[Corollary 1.9]{Conca} that the multiplicity is equal to the number of maximal NE-paths in the diagram. We collect this result as a theorem. 

\begin{theorem}\label{thm:multiplicity-paths}
 Let $W$ be a strongly stable set of monomials of degree two. Then $e(\K[W])$ is equal to the number of maximal NE-paths in the diagram representing $W$. 
\end{theorem}

For a diagram $L$ of a strongly stable set, we will use the notation $e(L)$ for the number of maximal NE-paths in $L$. We illustrate the key points from \cite{Conca} and \cite{CNPY} that provides the proof of Theorem \ref{thm:multiplicity-paths} in Example \ref{ex:multiplicity}.

\begin{ex}\label{ex:multiplicity}
\begin{figure}[ht]
\centering
 \includegraphics[scale=0.65]{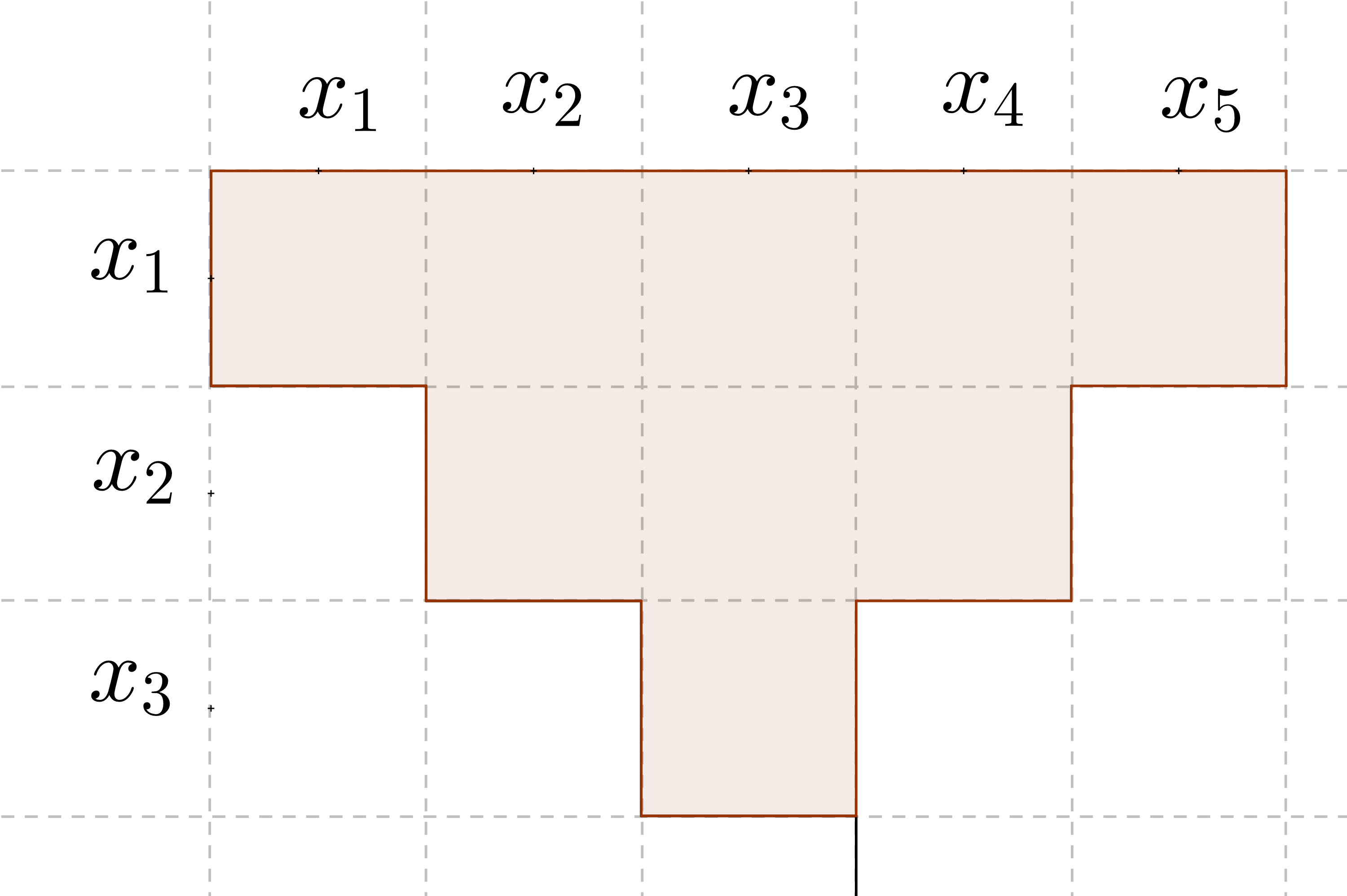}
 \caption{$ \st(x_1x_5,x_2x_4,x_3^2)$}
 \label{fig:ex_mult}
\end{figure}

Let $W=\st(x_1x_5,x_2x_4,x_3^2)$, the set in Figure \ref{fig:ex_mult}. Let
 \[ T=\K[y_{11},y_{12},y_{13},y_{14},y_{15},y_{22},y_{23},y_{24},y_{33}]\]
and consider the surjective homomorphism $\phi:T \to \K[W]$ defined by $y_{ij} \mapsto x_ix_j$. The kernel is given by 
\begin{align*}
 J=(&y_{11}y_{22}-y_{12}^2,y_{11}y_{23}-y_{12}y_{13},y_{11}y_{24}-y_{12}y_{14},y_{11}y_{33}-y_{13}^2,y_{12}y_{23}-y_{13}y_{22},y_{12}y_{24}-y_{14}y_{22},\\
 & y_{12}y_{33}-y_{13}y_{23},y_{13}y_{24}-y_{14}y_{23},y_{22}y_{33}-y_{23}^2),
\end{align*}
so $\K[W] \cong T/J$, and the Hilbert function of $\K[W]$, as we defined it, is the same as the Hilbert function of $T/J$ given the standard grading. The generating set given for $J$ is a Gröbner basis, under the Lex order with 
\[
 y_{11}>y_{12}>y_{13}>y_{14}>y_{15}>y_{22}>y_{23}>y_{24}>y_{33}.
\]
Hence the Hilbert function of $T/J$ is the same as the Hilbert function of 
\[
 T/\init(J) = T/(y_{11}y_{22},y_{11}y_{23},y_{11}y_{24},y_{11}y_{33},y_{12}y_{23},y_{12}y_{24},y_{12}y_{33},y_{13}y_{24},y_{22}y_{33}).
\]
This is the Stanley-Reisner ring with the facets 
\[
 y_{11}y_{12}y_{13}y_{14}y_{15}, y_{22}y_{12}y_{13}y_{14}y_{15}, y_{22}y_{23}y_{13}y_{14}y_{15}, y_{22}y_{23}y_{24}y_{14}y_{15},y_{33}y_{23}y_{13}y_{14}y_{15}, y_{33}y_{23}y_{24}y_{14}y_{15}.
\]
Notice that they all have the same dimension, and they correspond exactly to the maximal NE-paths of the diagram in Figure \ref{fig:ex_mult}. It is a known fact about Stanley-Reisner rings that the multiplicity is equal to the number of facets of maximal dimension, which here are exactly those listed above. 
\end{ex}
 
 For strongly stable sets in higher degrees, the ideal $J$ need not have a quadratic Gröbner basis. For this reason, Theorem \ref{thm:multiplicity-paths} does not generalize to higher degrees. 
 
Now, let us return to the questions 1 and 3 in the introduction.  
\begin{ex}\label{ex:persist_counterex}
 Let $n=12$, $d=2$ and $u=71$. There are five strongly stable sets of size 71, of monomials of degree two, in twelve variables, namely $W_1, \ldots, W_5$ illustrated in Figures \ref{fig:ex_persist1}-\ref{fig:ex_persist5}.
 
\begin{figure}[ht]
\begin{minipage}[t]{0.3\textwidth}
\includegraphics[scale=0.3]{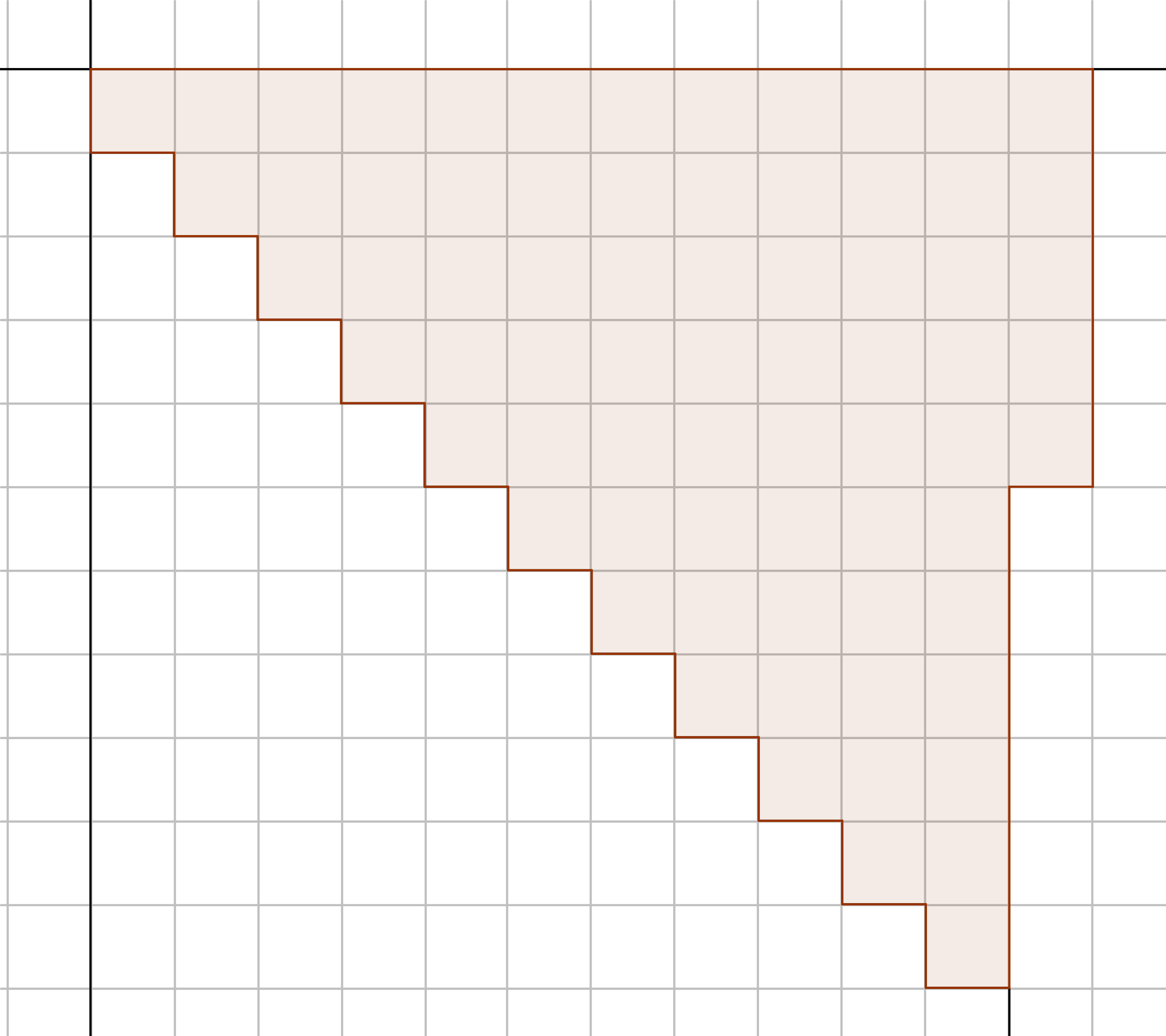}
\captionsetup{singlelinecheck=off}
\caption[.]{\vspace{-0.3cm}\begin{align*}W_1=\st(x_{11}^2,x_5x_{12})\end{align*}}\label{fig:ex_persist1}
\end{minipage}\begin{minipage}[t]{0.35\textwidth}
\includegraphics[scale=0.3]{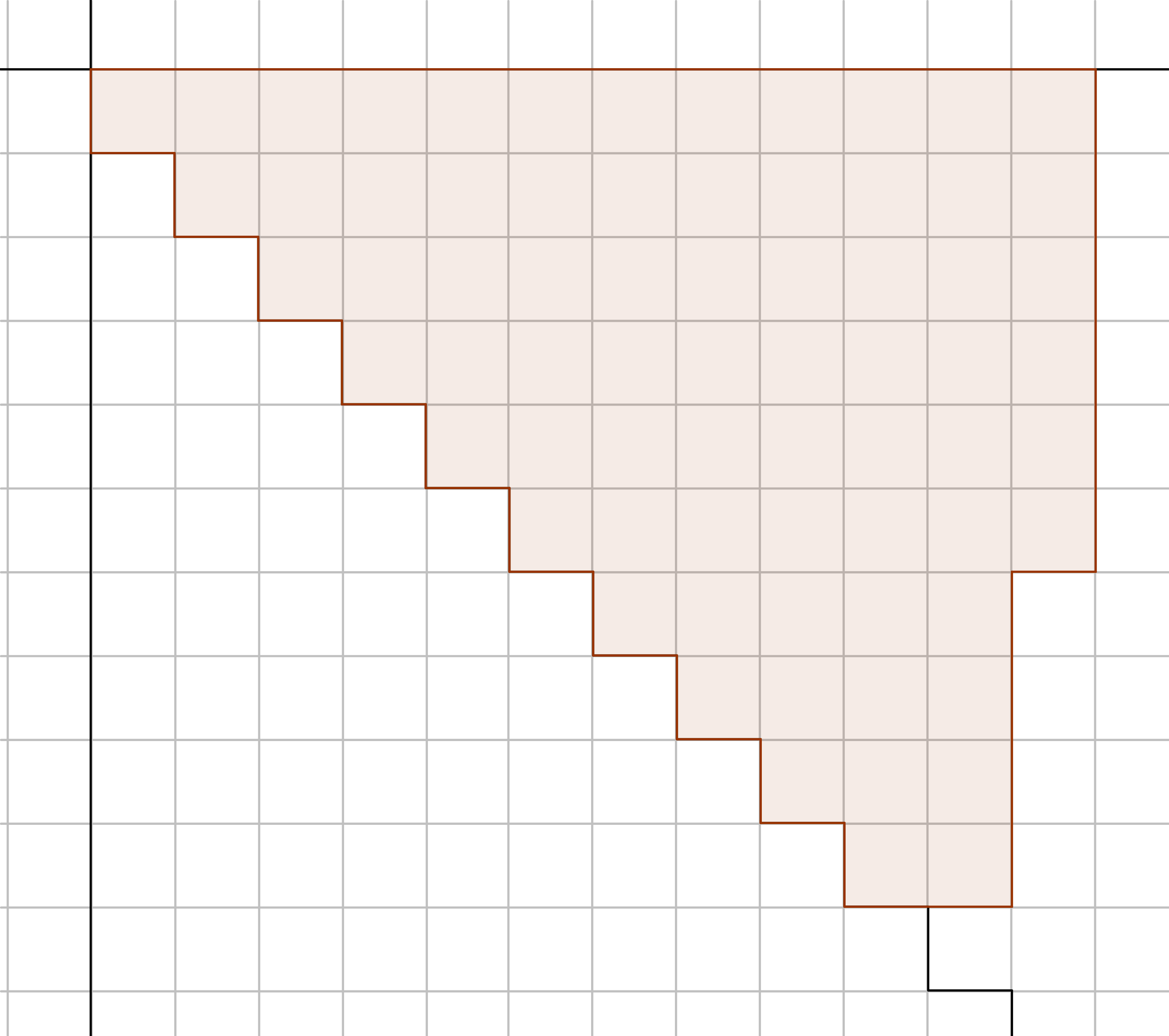}
\captionsetup{singlelinecheck=off}
\caption[.]{\vspace{-0.3cm}\begin{align*}W_2=\st(x_{10}x_{11},x_6x_{12})\end{align*}}\label{fig:ex_persist2}
\end{minipage}\begin{minipage}[t]{0.35\textwidth}
\includegraphics[scale=0.3]{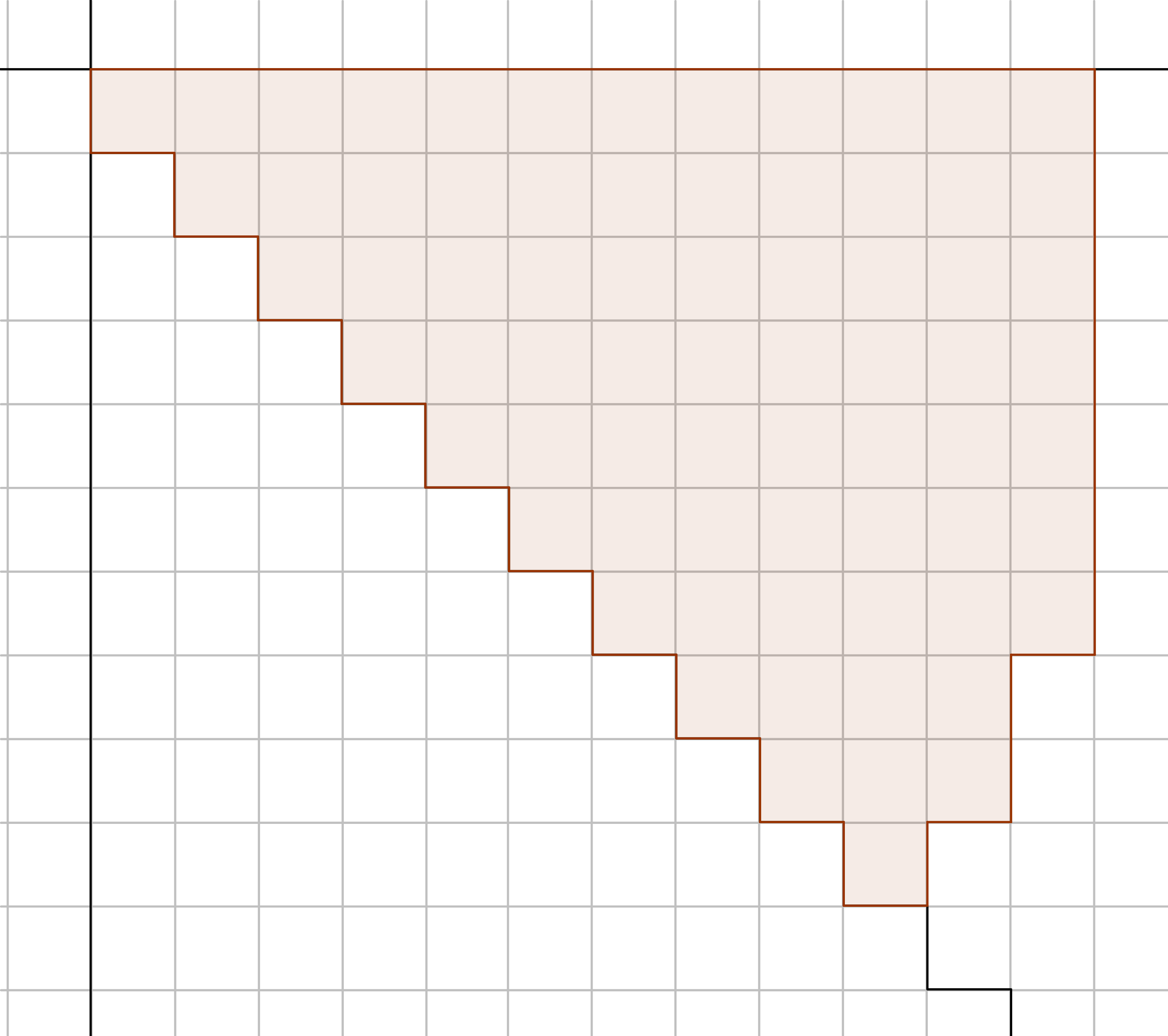}
\captionsetup{singlelinecheck=off}
\caption[]{\vspace{-0.3cm}\begin{align*}W_3=\st(x_{10}^2,x_9x_{11},x_7x_{12})\end{align*}}\label{fig:ex_persist3}
\end{minipage}

\begin{minipage}[t]{0.4\textwidth}
\includegraphics[scale=0.3]{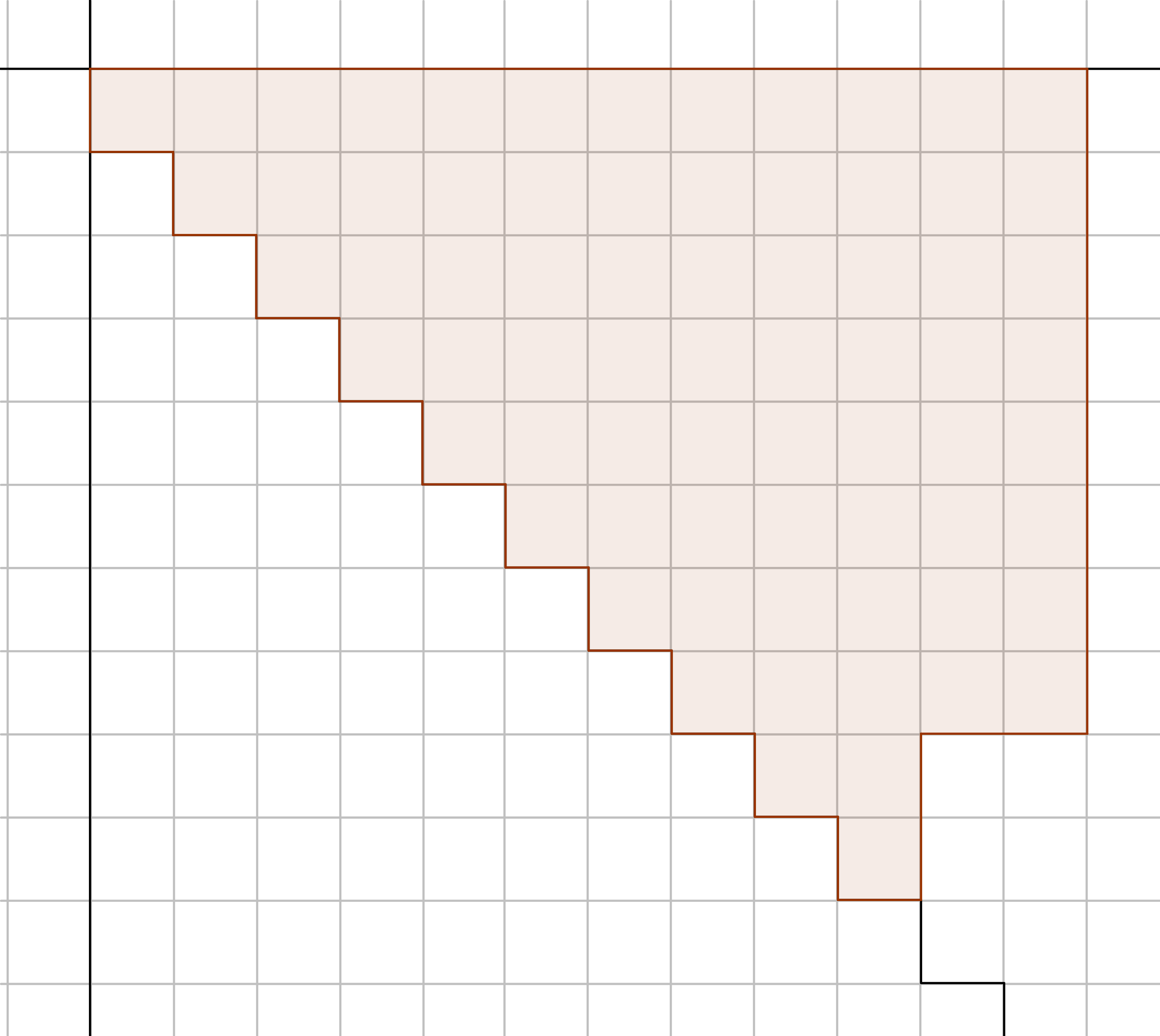}
\caption{$W_4=\st(x_{10}^2,x_8x_{12})$}\label{fig:ex_persist4}
\end{minipage}\begin{minipage}[t]{0.4\textwidth}
\includegraphics[scale=0.3]{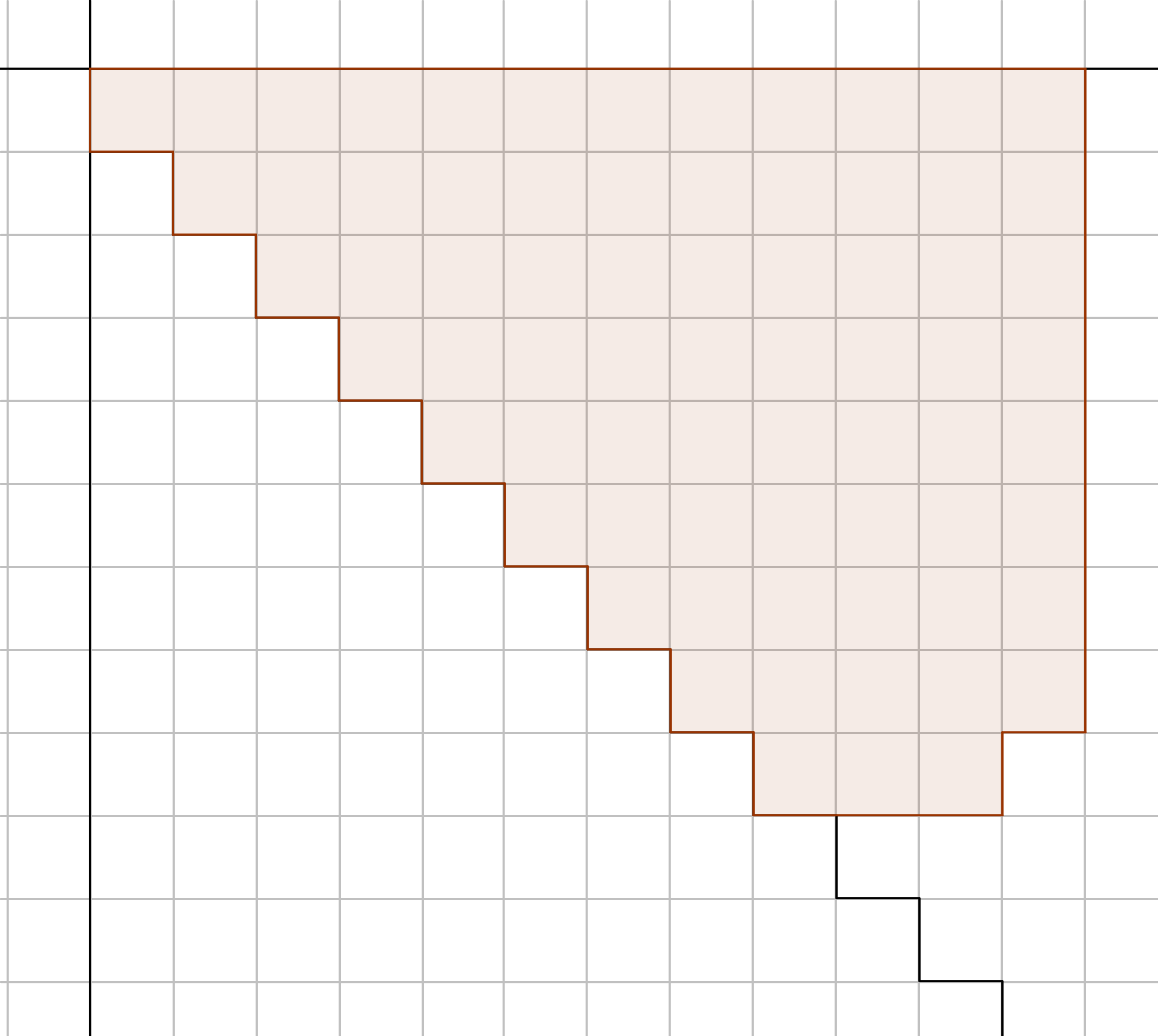}
\caption{$W_5 = \st(x_9x_{11},x_8x_{12})$}\label{fig:ex_persist5}
\end{minipage}
\end{figure}

To see that these are all the strongly stable sets, we may count the number of ways to remove seven boxes from the diagram of size 78 with 12 completely filled columns. If we remove seven boxes from the last column, we get $W_1$. If we remove boxes from the last two columns, there are three options which gives $W_2, W_3$ and $W_4$. There is only one way to remove seven boxes in the last three columns, which is $W_5$. Removing boxes from four or more columns will result in removing more than seven boxes. 

  A computation in Macaulay2 \cite{M2} gives
 \[
  e(\K[W_1])=1984, \ e(\K[W_2])=2010, \ e(\K[W_3])=2018, \ e(\K[W_4])=2008, \ e(\K[W_5])=1980,
 \]
which means that $\K[W_5]$ has the minimal Hilbert function, at least in the asymptotic sense. However, a computation of the Hilbert functions shows that $\HF(\K[W_5],i)$ is not minimal for $i=2$. The first $i$ for which  $\HF(\K[W_5],i)$ is minimal is $i=7$, as we can see in the following table. 

\begin{center}
\begin{tabular}{c|cccccc}
$i$          & 2     & 3     & 4     & 5    & 6         & 7\\
\hline
$\HF(\K[W_1],i)$ & \cellcolor{pink!45}1246 & \cellcolor{pink!45}11389 & \cellcolor{pink!45}70051 & \cellcolor{pink!45}328771 & \cellcolor{pink!45}1266005 & 4188859 \\
$\HF(\K[W_2],i)$ & 1256 & 11524 & 71012 & 333593 & 1285193 & 4253378\\
$\HF(\K[W_3],i)$ & 1259 & 11565 & 71306 & 335075 & 1291108 & 4273307\\
$\HF(\K[W_4],i)$ & 1255 & 11511 & 70922 & 333151 & 1283464 & 4247645\\
$\HF(\K[W_5],i)$ & 1248 & 11406 & 70124 & 328965 & 1266265 & \cellcolor{pink!45}4188404 
\end{tabular}
\end{center}
 This proves that the answer to the questions 1 and 3 is negative. 
 \end{ex}

\section{Subalgebras defined by Lex and RevLex segments}
Any initial segment of monomials of degree $d$, according to a monomial ordering in $\K[x_1, \ldots, x_n]$, is a strongly stable set. We will now focus on two monomial orderings, namely Lex and (graded) RevLex. In degree two, they may be defined as follows.

\begin{definition}
 Let $1 \le i \le j \le n$, and $1 \le k \le \ell \le n$. Then 
 \[x_ix_j >_{\Lex}x_kx_\ell \mbox{ if } i<k, \mbox{ or if } i=k \mbox{ and } j<\ell, \mbox{ and}\]
 \[x_ix_j >_{\RevLex}x_kx_\ell \mbox{ if } j<\ell, \mbox{ or if } j=\ell \mbox{ and } i<k.\]
\end{definition}
In terms of diagrams, we may say that $>_{\Lex}$ orders the monomials firstly by row, and secondly by column, and that $>_{\RevLex}$ orders the monomials firstly by column, and secondly by row. 
 
 Recall that, to minimize the degree of the Hilbert polynomial, we want to minimize the number of variables, with respect to the given $u$.  In terms of the diagram, we want the diagram to be such that we can not draw another diagram with the same number of boxes, but fewer columns. Since there are $\binom{n+1}{2}$ monomials of degree two in $n$ variables, we choose $n$ so that $\binom{n}{2} <u \le \binom{n+1}{2}$. Since $\binom{n+1}{2}-\binom{n}{2}=n$ we may write $u=\binom{n}{2}+r$ for some $0<r \le n$, or  $u=\binom{n+1}{2}-s$ for some $0\le s < n$. 

 \begin{definition}
  For a positive integer $u$, let $n$ be the unique number such that $\binom{n}{2} <u \le \binom{n+1}{2}$. Then we let $\Lex(u)$ be the set of the $u$ greatest monomials of degree two, according to $>_{\Lex}$. Similarily, we let $\RevLex(u)$ be the set of the $u$ greatest monomials of degree two, according to $>_{\RevLex}$. 
 \end{definition}
 
 The diagram in Figure \ref{fig:ex_persist1} is $\RevLex(71)$ and the diagram in Figure \ref{fig:ex_persist5} is $\Lex(71)$. 
 
\begin{remark}\label{rmk:large_u}
 For $u=\binom{n+1}{2}-s$ and  $s=0,1,$ or $2$ there is only one strongly stable set of size $u$ in $n$ variables, and this set is both a Lex and a RevLex segment. See Figure \ref{fig:large_u} for an example.
\end{remark}

\begin{figure}[ht]
 \begin{minipage}{0.3\textwidth}
  \includegraphics[scale=0.3]{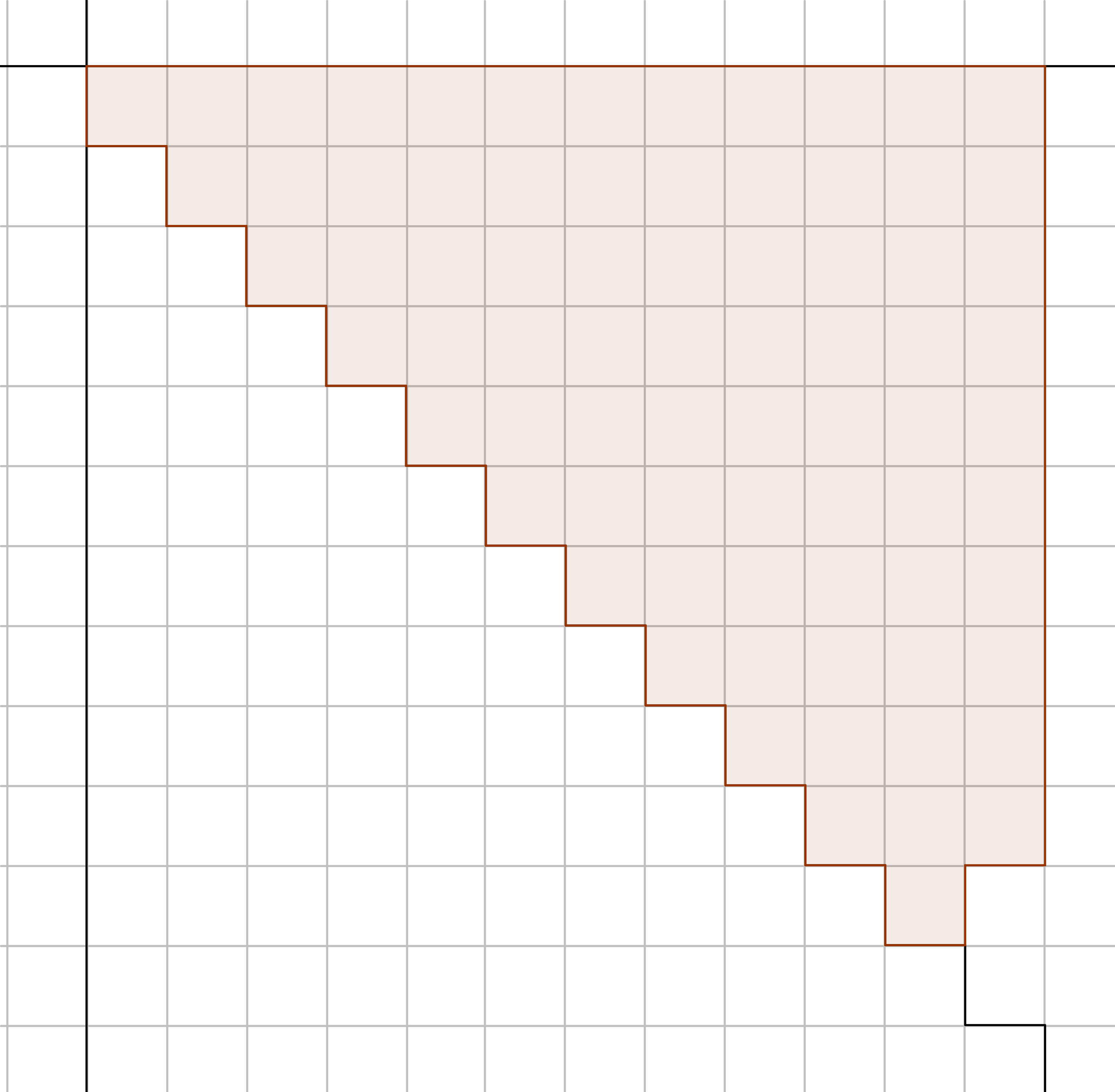}
 \end{minipage}\begin{minipage}{0.3\textwidth}
  \includegraphics[scale=0.3]{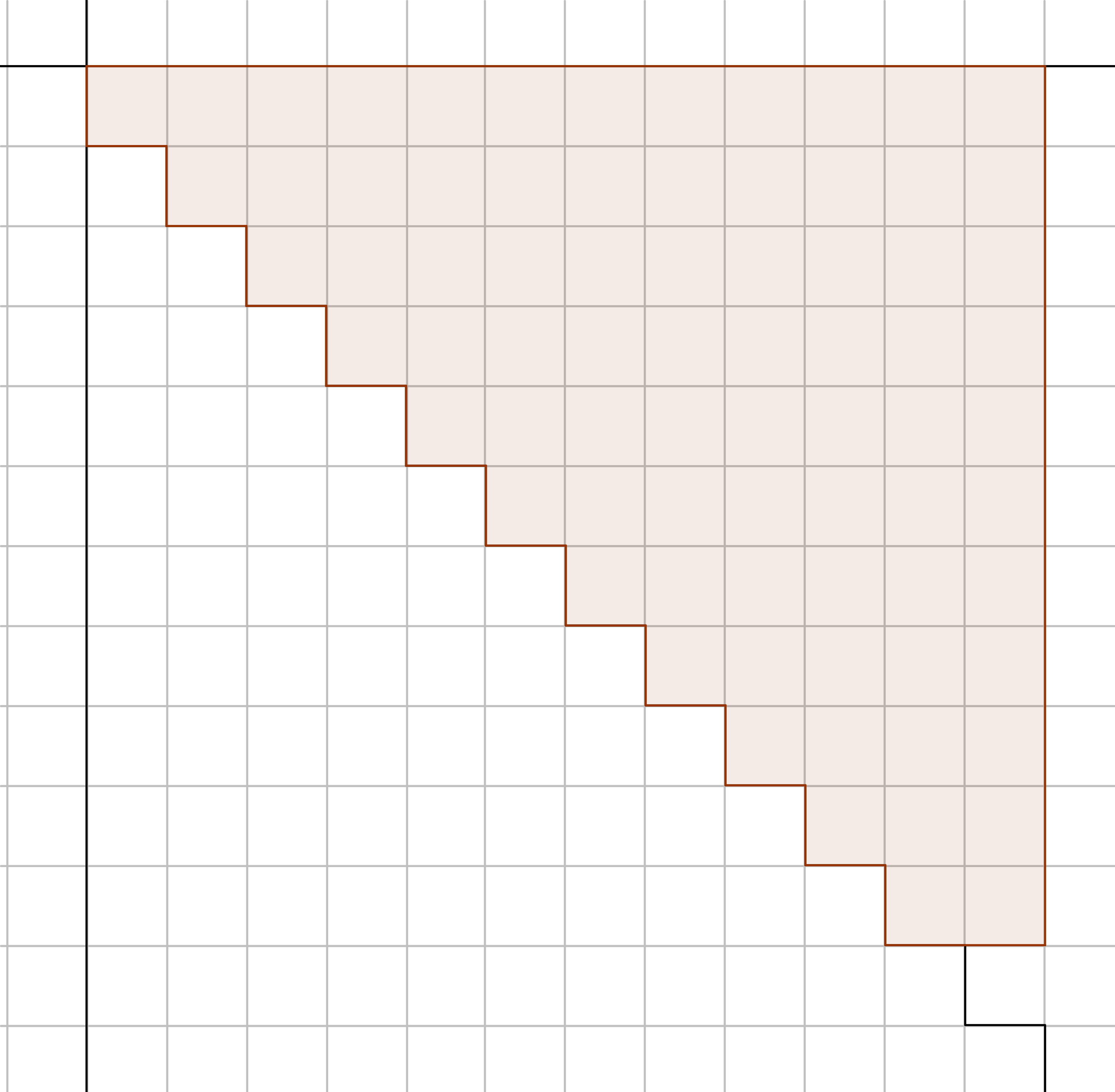}
 \end{minipage}\begin{minipage}{0.3\textwidth}
  \includegraphics[scale=0.3]{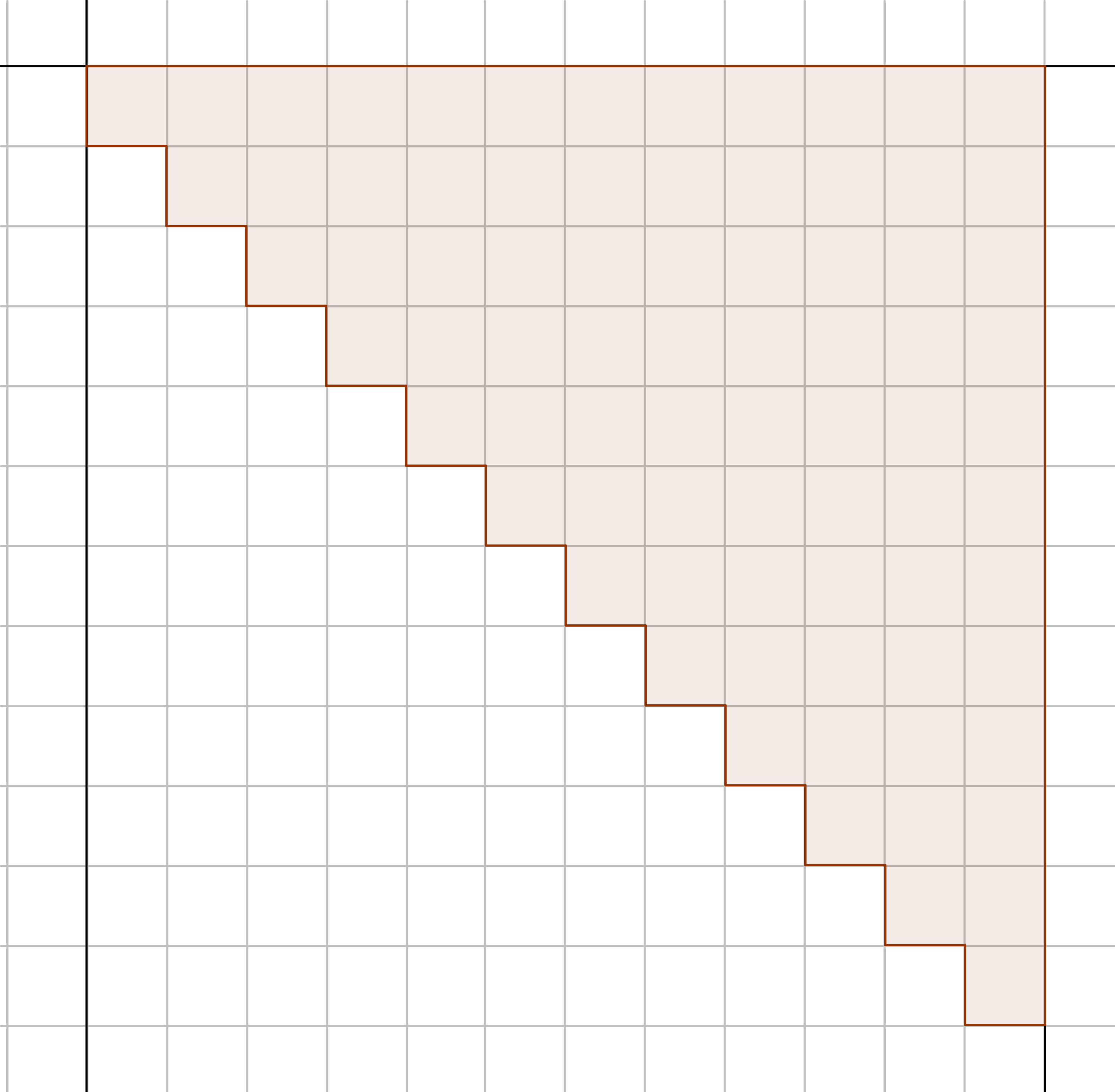}
 \end{minipage}\caption{The only strongly stable sets of sizes 76, 77, and 78 in 12 variables.}\label{fig:large_u}
\end{figure}

Notice that the subalgebra with the minimal Hilbert function in Example \ref{ex:persist_counterex} was generated by the set $\Lex(71)$, and the ``competition'' was between $\Lex(71)$ and $\RevLex(71)$.

\begin{conjecture}\label{conj}
 One of the algebras $\K[\Lex(u)]$ or $\K[\RevLex(u)]$ has the minimal Hilbert function, for a subalgebra of $\K[x_1, \ldots, x_n]$ generated by $u$ forms of degree two. 
\end{conjecture}

Conjecture \ref{conj} can also be phrased as a purely combinatorial statement, see Appendix \ref{app:comb_conj}.

\begin{remark}\label{rmk:data}
 Conjecture \ref{conj} is true for $n \le 80$, which means $u \le 3240$. This is proved by a computation of the multiplicities in Mathematica \cite{Math}. The results of the computation, as well as a description of how the computation was made, can be found in Appendix \ref{app:data}.
\end{remark}

For $u=7$ and $24$ the sets $\Lex(u)$ and $\RevLex(u)$ give the same Hilbert polynomial. For $u=40$ the sets $\Lex(u)$ and $\RevLex(u)$ give the same multiplicity. Computing their Hilbert polynomials in Macaulay2 gives 
\begin{align*}
 8!&\HF(\K[\Lex(40)],i) =  \\
 &240i^8  + 4248i^7  + 31640i^6  + 129192i^5  + 315560i^4  + 471072i^3  + 418640i^2  + 201888i + 40320
\end{align*}
and 
\begin{align*}
 8!&\HF(\K[\RevLex(40)],i) =\\
 &240i^8  + 4256i^7  + 31752i^6  + 129752i^5  + 316680i^4  + 471464i^3  + 417408i^2  + 200928i + 40320
\end{align*}
and we can see that the Lex ordering gives the minimal Hilbert function, as $\HF(\K[\Lex(40)],i)$ has the smaller coefficient for $i^7$. For $n \le 80$ the values $u=7, 24,$ and $40$ are the only values of $u$ for which $\Lex(u)$ and $\RevLex(u)$ give the same multiplicity, apart from $u=\binom{n+1}{2}-s$ with $s=0,1,2$, which we saw in Remark \ref{rmk:large_u}.

Conjecture \ref{conj} only states that the algebras with minimal Hilbert function are given by Lex or RevLex segments, it does not tell us which of the two orderings it is for a given $u$. One direct way to find out is of course to compute the multiplicities explicitly. 

\begin{lemma}\label{lemma:mult_revlex}
 For $\binom{n}{2} <u \le \binom{n+1}{2}$, we have
 \[
  e(\K[\RevLex(u)])= \begin{cases} 
  2^{n-1}-2^{n-r-1} & \mbox{for} \ r<n \\
  2^{n-1} & \mbox{for} \  r=n,
  \end{cases}
 \]
 where $0 <r \le n$ such that $u=\binom{n}{2}+r$.
\end{lemma}
\begin{proof}
 The total number of maximal NE-paths from the diagonal to the upper right corner of the diagram is $2^{n-1}$, since the paths have length $n-1$. If $r=n$ all these paths are inside the diagram. In the case $r<n$ we must subtract the number of paths that goes outside the RevLex-diagram. Those are exactly the paths going through the $x_{r+1}x_n$-box, which is $n-r-1$ steps from the diagonal. This gives us the formula $2^{n-1}-2^{n-r-1}$.
\end{proof}

\begin{lemma}\label{lemma:mult_lex}
 For $\binom{n}{2} <u \le \binom{n+1}{2}$, let $k$ be the largest integer such that $\st(x_{n-k}x_n) \supseteq \Lex(u)$. If $\st(x_{n-k}x_n) = \Lex(u)$, then 
 \[
  e(\K[\Lex(u)]) = \sum_{i=k}^{n-1} \binom{n-1}{i},
 \]
 and otherwise
 
 \[
  e(\K[\Lex(u)]) = \sum_{i=k}^{n-1} \binom{n-1}{i} - \binom{n-k+j-2}{j-1}
 \]
 where $j=|\st(x_{n-k}x_n)| - |\Lex(u)|$.
\end{lemma} 
\begin{proof}
 Let us first compute the number of maximal NE-paths in the diagram of $\st(x_{n-k}x_n)$. The number of maximal NE-paths starting in a row of length $i$ is $\binom{n-1}{i-1}$, as such a path is of length $n-1$ and should have precisely $i-1$ steps right. In $\st(x_{n-k}x_n)$ the first row has length $n$ and the last length $k+1$, so we get $\sum_{i=k}^{n-1} \binom{n-1}{i}$ paths. Now we must subtract those paths that are not inside the diagram of $\Lex(u)$. Those are precisely the paths that go through the $x_{n-k}x_{n-j+1}$-box, i. e. the first box in the last row that is not contained in $\Lex(u)$. There is only one way to go from the diagonal to $x_{n-k}x_{n-j+1}$. From there $n-k+j-2$ steps remains, and $j-1$ of those should be steps right. Hence we should subtract $\binom{n-k+j-2}{j-1}$, which proves the formula. 
\end{proof}

Even with the formulas for the multiplicities given in Lemma \ref{lemma:mult_revlex} and Lemma \ref{lemma:mult_lex}, it is not obvious which one gives the smaller value for a given $u$. Looking at the data in Appendix \ref{app:data}, the pattern is still not completely clear. However, two observations can be made. 

\begin{enumerate}
 \item For a given $n$, there are at most three shifts between Lex and RevLex. This has been confirmed by computation for $n \le 1000$.
 \item We have RevLex for small $r$, and Lex for large $r$ in the interval $1 \le r \le n$. For $n \ge 80$ we will see in Theorem \ref{thm:lex} that $\Lex(\binom{n}{2}+r)$ gives the minimal Hilbert function for $n-25 \le r \le n$, and in Theorem \ref{thm:revlex} that $\RevLex(\binom{n}{2}+r)$ gives the minimal Hilbert function for $1 \le r \le 50$. 
\end{enumerate}
 We summarize the cases where Conjecture \ref{conj} is proved in a theorem.

\begin{theorem}\label{thm:main}
 One of the algebras $\K[\Lex(u)]$ or $\K[\RevLex(u)]$ has the minimal Hilbert function, for a subalgebra of $\K[x_1, \ldots, x_n]$ generated by $u$ forms of degree two, in the following cases.
 \begin{itemize}
  \item $n \le 80$, in which case the algebras are listed in Appendix \ref{app:data},
  \item $u=\binom{n}{2}+r$ with $n \ge 80$ and $1 \le r \le 50$, in which case it is given by $\RevLex(u)$,
  \item  $u=\binom{n}{2}+r$ with $n \ge 80$ and $n-25 \le r \le n$, in which case it is given by $\Lex(u)$.
 \end{itemize}
\end{theorem}

\subsection{Lex segments}\label{sec:Lex}
In this section we will focus on algebras generated by sets $\Lex(u)$, typically for large $u$ in the interval $\binom{n}{2}<u\le \binom{n+1}{2}$. In this setting it is convenient to use the representation $u=\binom{n+1}{2}-s$. 

\begin{prop}\label{prop:lex}
Let $S$ and $n$ be fixed integers such that $0 \le S < n$. Suppose that for all $0 \le s \le S$, the algebra $\K[\Lex(\binom{n+1}{2}-s)]$ has minimal multiplicity, among all subalgebras of $\K[x_1, \ldots, x_n]$ generated by $\binom{n+1}{2}-s$ forms of degree two. 
Then $\K[\Lex(\binom{n+2}{2}-s)] \subset \K[x_1, \ldots, x_{n+1}]$ has minimal multiplicity for all $0\le s \le S$.
\end{prop}
\begin{proof}
 Let $L$ be the diagram of a strongly stable set in $\K[x_1, \ldots, x_{n+1}]$ of size $\binom{n+2}{2}-s$, for some $s \le S$. Let $L'$ be the diagram obtained from $L$ be removing the top row, and let $L''$ be the diagram obtained from $L$ by removing the last column. Both $L'$ and $L''$ are diagrams of strongly stable sets in $\K[x_1, \ldots, x_n]$. $L'$ is of size $\binom{n+1}{2}-s$ and $L''$ is of size $\binom{n+1}{2}-s'$ with $s'\le s$. All maximal NE-path in $L$ ending with a step up can be considered maximal NE-paths in $L'$ by removing the last step. In the same way all maximal NE-paths in $L$ ending with a step right can be considered maximal NE-paths in $L''$. It follows that $e(L)=e(L')+e(L'')$. Applying the same argument to $\Lex(\binom{n+2}{2}-s)$ we get 
\[   e(\K[\Lex\left(\tbinom{n+2}{2}-s\right)])= e(\K[\Lex \left(\tbinom{n+1}{2}-s\right)])+e(\K[\Lex\left(\tbinom{n+1}{2}-s''\right)])  \]
 with $s''\le s$. Notice that $s''\ge s'$ as the last column in the diagram of $\Lex(\binom{n+2}{2}-s)$ has at least as many boxes as the last column in $L$.   
 
 By assumption we know that $e(L') \ge e(\K[\Lex(\binom{n+1}{2}-s)])$ and $e(L') \ge e(\K[\Lex(\binom{n+1}{2}-s')])$. We now have
 \begin{align*}
 e(L) = e(L')+e(L'')& \ge e(\K[\Lex\left(\tbinom{n+1}{2}-s\right)])+ e(\K[\Lex\left(\tbinom{n+1}{2}-s'\right)])\\ & \ge e(\K[\Lex \left(\tbinom{n+1}{2}-s\right)])+e(\K[\Lex\left(\tbinom{n+1}{2}-s''\right)])  = e(\K[\Lex\left(\tbinom{n+2}{2}-s\right)])
 \end{align*}
 and thus we have proved that $e(\K[\Lex\left(\tbinom{n+2}{2}-s\right)]) \le e(L)$ for any $L$ of size $\tbinom{n+2}{2}-s$, with $s \le S$.  
\end{proof}

\begin{theorem}\label{thm:lex}
 Let $u=\binom{n+1}{2}-s$, where $n \ge 80$ and $s\le 25$. Then $\K[\Lex(u)]$ has the minimal Hilbert function, among all subalgebras of $\K[x_1, \ldots, x_n]$ generated by $u$ forms of degree two.
\end{theorem}
\begin{proof}
 For $n=80$, see Appendix \ref{app:data}. If follows inductively from Proposition \ref{prop:lex} that it also holds for all $n>80$.  
\end{proof}

\subsection{RevLex segments}\label{sec:RevLex}
We will now study algebras generated by sets $\RevLex(u)$, where $u=\binom{n}{2}+r$ and $r$ is small. 

\begin{lemma}\label{lemma:revlex_step}
 Let $u_n=\binom{n}{2}+r$ for some fixed $r$ with $1 \le r \le n$. Suppose $\K[\RevLex(u_n)]$ has minimal multiplicity among all subalgebras of $\K[x_1, \ldots, x_n]$ generated by $u_n$ forms of degree two. If $\Lex(u_{n+1})=\st(g)$ for a monomial $g$, then either $\K[\Lex(u_{n+1})]$ or $\K[\RevLex(u_{n+1})]$ has minimal multiplicity among subalgebras of $\K[x_1, \ldots, x_{n+1}]$ generated by $u_{n+1}$ forms. If $\Lex(u_{n+1})\ne \st(g)$, then $\K[\RevLex(u_{n+1})]$ has minimal multiplicity. 
\end{lemma}
\begin{proof}
 Let $L$ be the diagram of some strongly stable set of size $u_{n+1}$, which has more than one strongly stable generator. Let $L'$ be the diagram obtained by removing the boxes on the last row of $L$, except the first box, i. e. the one on the diagonal. Clearly $e(L) \ge e(L')$. Let $L''$ be the diagram obtained by removing all boxes on the diagonal of $L'$. This is also a diagram of a strongly stable set, after a shift in the row and column indices. We have $e(L')=2e(L'')$, since every maximal NE-path in $L'$ comes from adding an up och right step to the beginning to a path of $L''$. Notice that the boxes we have removed from $L$ in the two steps were all in different columns. We have not removed any box from from the last column, since $L$ had more than one strongly stable generator. This means that $|L|-|L''|\le n$, and hence $|L''| \ge u_{n+1}-n=u_n$. Let $L'''$ be a diagram obtained from $L''$ by, if necessary, removing some arbitrary boxes so that $|L'''|=u_n$. It is clearly possible to do this in such a way so that $L'''$ is still a valid diagram for a strongly stable set. By assumption $e(L''') \ge e(\K[\RevLex(u_n)]) .$ We now have
 \[
  e(L)\ge e(L')=2e(L'') \ge 2e(L''') \ge 2e(\K[\RevLex(u_n)]) = e(\K[\RevLex(u_{n+1})])
 \]
 where the last equality follows from Lemma \ref{lemma:mult_revlex}.
 
 We have now proved that the strongly stable set of size $u_{n+1}$ which gives minimal multiplicity is either $\RevLex(u_{n+1})$ or $\st(g)$ for some monomial $g$. As $\st(g)$ is a Lex segment, the proof is complete.   
\end{proof}

As we can see in Lemma \ref{lemma:revlex_step}, the situation is a bit more complicated than for the Lex-algebras in Section \ref{sec:Lex}. Lemma \ref{lemma:revlex_step} can not be used directly as an induction step, we need to analyze the situation when $\Lex(u_n)=\st(g)$ further. With $u_n=\binom{n}{2}+r$, and $r$ fixed, for which $n$ does this situation occur? The monomial $g$ has to be divisible by $x_n$, so we have $g=x_{n-k}x_n$ for some number $k$.  The monomials of degree two \emph{not} in $\st(x_{n-k}x_n)$ are the $\binom{k+1}{2}$ monomials in the variables $x_{n-k+1}, \ldots, x_n$. It follows that we can write $u_n=\binom{n+1}{2}-\binom{k+1}{2}$. Then $\binom{n}{2}+r=\binom{n+1}{2}-\binom{k+1}{2}$, and it follows that $n=\binom{k+1}{2}+r$. To summarize, $\Lex(u_n)=\st(x_{n-k}x_n)$ precisely when $n=\binom{k+1}{2}+r$.

Our next goal is to prove that if $\Lex(u_n)=\st(x_{n-k}x_n)$ for some $k$, and $e(\Lex(u_n))>e (\RevLex(u_n))$, then $e(\Lex(u_{n'}))>e (\RevLex(u_{n'}))$ for all $n'\ge n$. To do this, we first need the following technical lemma.  

\begin{lemma}\label{lemma:revlex_binomsum}
 Let $k$ and $s$ be integers, $s \ge 0$ and $k \ge 3$, and let $m=\binom{k+1}{2}+s$. If $k$ is large enough compared to $s$, so that $(k-1)(2^k-\frac{3}{2}k)-1>s$, then 
 \[
  \sum_{i=k}^m \binom{m}{i} >2^k \sum_{i=k-1}^{m-k} \binom{m-k}{i}.
 \]
\end{lemma}
\begin{proof}
 Define
 \[
  F(t)=2^t \sum_{i=k}^{m-t}\binom{m-t}{i}+(2^t-1)\binom{m-t}{k-1}+(2^t-t-1)\binom{m-t}{k-2}
 \]
 for $t \le m-k = \binom{k}{2}+s$. Note that $F(k)$ is defined, since $k \le \binom{k}{2}+s$ holds for all $k \ge 3$. The idea of the proof is to show that the two inequalities
 \begin{equation}\label{ineq1:lemma_binomsum}
   \sum_{i=k}^m \binom{m}{i} \ge F(k) >2^k \sum_{i=k-1}^{m-k} \binom{m-k}{i}
 \end{equation}
hold. As $F(0)= \sum_{i=k}^m \binom{m}{i}  $, the first inequality is $F(0) \ge F(k)$. We will prove this by showing that $F(t) \ge F(t+1)$. Recall that 
\[
 \binom{m-t}{i} = \binom{m-(t+1)}{i} + \binom{m-(t+1)}{i-1}
\]
for $0 <i<m-t$. We get
\begin{align*}
 \sum_{i=k}^{m-t} \binom{m-t}{i}&= \sum_{i=k}^{m-(t+1)} \binom{m-t}{i} +1 \\
 &=  \sum_{i=k}^{m-(t+1)}\left(\binom{m-(t+1)}{i} + \binom{m-(t+1)}{i-1}\right) +\binom{m-(t+1)}{m-(t+1)} \\
 &=  \sum_{i=k}^{m-(t+1)}\binom{m-(t+1)}{i}+  \sum_{i=k-1}^{m-(t+1)}\binom{m-(t+1)}{i}\\
 &= 2\sum_{i=k}^{m-(t+1)}\binom{m-(t+1)}{i} + \binom{m-(t+1)}{k-1} .
\end{align*}
From this we obtain
\begin{align*}
 F(t)=&2^t \sum_{i=k}^{m-t}\binom{m-t}{i}+(2^t-1)\binom{m-t}{k-1}+(2^t-t-1)\binom{m-t}{k-2}\\
 =&2^{t+1} \sum_{i=k}^{m-(t+1)}\binom{m-(t+1)}{i}+2^t\binom{m-(t+1)}{k-1}+  \\
 & \quad + (2^t-1)\left(\binom{m-(t+1)}{k-1}+\binom{m-(t+1)}{k-2}\right)  + \\
 & \quad +(2^t-t-1)\left(\binom{m-(t+1)}{k-2}+\binom{m-(t+1)}{k-3}\right) \\
 \ge &  2^{t+1} \sum_{i=k}^{m-(t+1)}\!\!\! \binom{m-(t+1)}{i} + (2^{t+1}-1)\binom{m-(t+1)}{k-1} + (2^{t+1}-(t+1)-1)\binom{m-(t+1)}{k-2} \\
 =&F(t+1).
\end{align*}
We have now proved the first inequality of (\ref{ineq1:lemma_binomsum}). If 
\begin{equation}\label{ineq2:lemma_binomsum}
 2^k-k-1>\frac{m-2k+2}{k-1}
\end{equation}
 it follows that
\begin{align*}
 F(k) =& 2^k \sum_{i=k}^{m-k}\binom{m-k}{i}+(2^k-1)\binom{m-k}{k-1}+(2^k-k-1)\binom{m-k}{k-2}\\
 >& 2^k \sum_{i=k}^{m-k}\binom{m-k}{i}+(2^k-1)\binom{m-k}{k-1}+  \frac{m-2k+2}{k-1} \binom{m-k}{k-2}\\
 =& 2^k \sum_{i=k}^{m-k}\binom{m-k}{i}+(2^k-1)\binom{m-k}{k-1}+   \binom{m-k}{k-1} = 2^k \sum_{i=k-1}^{m-k} \binom{m-k}{i},
\end{align*}
which is the second inequality of (\ref{ineq1:lemma_binomsum}). Hence we need to verify (\ref{ineq2:lemma_binomsum}).
Since $m=\binom{k+1}{2}+s=\frac{k(k+1)}{2}+s$, (\ref{ineq2:lemma_binomsum}) is equivalent to 
\[
 2^k>\frac{\frac{k(k+1)}{2}+s-2k+2 }{k-1}+k+1,
\]
and the right hand side simplifies to $\frac{3}{2}k+\frac{s+1}{k-1}$. Hence
\[
 (\ref{ineq2:lemma_binomsum}) \ \iff \ 2^k>\frac{3}{2}k+\frac{s+1}{k-1} \ \iff \ (k-1)(2^k-\frac{3}{2}k)-1>s,
\]
which is true by assumption. We have now proved both inequalities of (\ref{ineq1:lemma_binomsum}). 
\end{proof}

\begin{lemma}\label{lemma:revlex_final}
Let $k_0$ and $r$ be integers such that $k_0 \ge 3$, $r \ge 1$, and $k_0$ large enough compared to $r$ so that $(k_0-1)(2^{k_0}-\frac{3}{2}k_0)>r$. Let $u_n=\binom{n}{2}+r$. If $\K[\RevLex(u_n)]$ has minimal multiplicity among all subalgebras of $\K[x_1, \ldots, x_n]$ generated by $u_n$ forms of degree two, when $n=\binom{k_0}{2}+r$, then the same holds for all $n \ge \binom{k_0}{2}+r$. 
\end{lemma}
\begin{proof}
 If we can prove $e(\K[\RevLex(u_n)])<e(\K[\Lex(u_n)])$ for all $n \ge \binom{k_0}{2}+r$ such that $\Lex(u_n)$ has only one strongly stable generator, then we are done by Lemma \ref{lemma:revlex_step}. That is, we want to prove
 \begin{equation}\label{eq:lemma_revlex_final}
  e(\K[\RevLex(u_n)])<e(\K[\Lex(u_n)]) \ \ \mbox{for all} \ \ n=\binom{k}{2}+r, \ \ k \ge k_0.
 \end{equation}
This is true for $k=k_0$, by assumption. The proof proceeds by induction. We assume that (\ref{eq:lemma_revlex_final}) is true for some $k \ge k_0$, and we want to prove it for $k+1$. Let $n=\binom{k+1}{2}+r$, and notice that $\binom{k}{2}+r=n-k$. Applying Lemma \ref{lemma:mult_revlex} and Lemma \ref{lemma:mult_lex} for the multiplicities, we are assuming that
\[
 2^{n-k-1}-2^{n-k-r-1}<\sum_{i=k-1}^{n-k-1}\binom{n-k-1}{i},
\]
and we want to prove
\[
 2^{n-1}-2^{n-r-1}< \sum_{i=k}^{n-1}\binom{n-1}{i}.
\]
By the inductive hypothesis we get
\[
 2^{n-1}-2^{n-r-1}=2^k(2^{n-k-1}-2^{n-k-r-1})<2^k\sum_{i=k-1}^{n-k-1}\binom{n-k-1}{i}.
\]
As $(k-1)(2^{k}-\frac{3}{2}k)>r$ holds for any $k \ge k_0$ by the assumption on $k_0$, we can apply Lemma \ref{lemma:revlex_binomsum} with $m=n-1$ and $s=r-1$. This gives 
\[
 2^k\sum_{i=k-1}^{n-k-1}\binom{n-k-1}{i}<\sum_{i=k}^{n-1}\binom{n-1}{i}
\]
and we are done. 
\end{proof}

Finally we can apply Lemma \ref{lemma:revlex_final} to get a class of minimal RevLex-algebras not included in the table in Appendix \ref{app:data}. 

\begin{theorem}\label{thm:revlex}
 Let $u_n=\binom{n}{2}+r$ for $n \ge 80$ and $1 \le r \le 50$. Then $\K[\RevLex(u_n)]$ has the minimal Hilbert function among the subalgebras of $\K[x_1, \ldots, x_n]$ generated by $u_n$ forms of degree two. 
\end{theorem}
\begin{proof}
 We will use Lemma \ref{lemma:revlex_final} with $k_0=9$. As $(k_0-1)(2^{k_0}-\frac{3}{2}k_0-1)=3988>50$ Lemma \ref{lemma:revlex_final} can indeed be applied for $1 \le r \le 50$, but let us first consider $1 \le r \le 44$. In the table in Appendix \ref{app:data} we see that $\K[\RevLex(u_n)]$  gives the minimal multiplicity for all $n=\binom{9}{2}+r=36+r$, i.\,e. $37 \le n \le 80$. By Lemma \ref{lemma:revlex_final} $\K[\RevLex(u_n)]$ will have the minimal multiplicity for all $n \ge 80$. 
 
 Next, let us consider $45 \le r \le 50$. It follows from Appendix \ref{app:data} and Lemma \ref{lemma:revlex_step}, with $n=80$, that $\RevLex(u_n)$ gives the minimal Hilbert function for $80 \le n <36+r$, as $n=36+r$ is the least $n>80$ for which $\Lex(u_n)=\st(g)$ for some monomial $g$. For $n=36+r$, Lemma \ref{lemma:revlex_step} only tells us that Lex or RevLex gives the minimal Hilbert function. Using Lemma \ref{lemma:mult_revlex} and Lemma \ref{lemma:mult_lex} we can compute $e(\K[\Lex(u_n)])$ and $e(\K[\RevLex(u_n)])$ for all $n=36+r$ with $45 \le r \le 50$, and verify that $e(\K[\RevLex(u_n)])$ has the smaller value. By Lemma \ref{lemma:revlex_final}, this holds also for any $n>36+r$, and we are done.    
\end{proof}

Theorem \ref{thm:lex} and Theorem \ref{thm:revlex} together with the data in Appendix \ref{app:data} now proves Theorem \ref{thm:main}.

\section{Concluding remarks}\label{sec:last}

A next step would be to look for a generalization of Conjecture \ref{conj} to higher degrees. Example 3.5 in \cite{Boij-Conca} shows that the minimal Hilbert function for a subalgebra of $\K[x_1,x_2,x_3]$ generated by 12 forms of degree five is not given by the Lex or RevLex segment. Hence, Conjecture \ref{conj} does not generalize directly to higher degrees, one needs to use other monomial orderings. In fact, this can be observed already in degree three. 

\begin{ex}\label{ex:deg3.1}
 For $n=4$, $d=3$, and $u=13$ there are eight strongly stable sets, namely
 \begin{align*}
W_1&=\st({x}_{1}{x}_{3}{x}_{4},\,{x}_{3}^{3}),& \ W_2&=\st({x}_{2}^{2}{x}_{4},\,{x}_{3}^{3}),\\
W_3&=\st({x}_{1}{x}_{4}^{2},\,{x}_{2}^{2}{x}_{4}),& \ W_4&=\st({x}_{1}{x}_{3}{x}_{4},\,{x}_{2}^{2}{x}_{4},\,{x}_{2}{x}_{3}^{2}),\\
W_5&=\st({x}_{1}{x}_{3}^{2},\,{x}_{2}^{2}{x}_{4},\,{x}_{3}^{3}),& \ W_6&=\st({x}_{1}{x}_{2}{x}_{4},\,{x}_{1}{x}_{3}^{2},\,{x}_{2}^{2}{x}_{4},\,{x}_{3}^{3}),\\
W_7&=\st({x}_{1}^{2}{x}_{4},\,{x}_{1}{x}_{3}^{2},\,{x}_{2}^{2}{x}_{4},\,{x}_{3}^{3}),& \ W_8&=\st({x}_{1}{x}_{4}^{2},\,{x}_{2}{x}_{3}^{2}).\\
 \end{align*}
These sets are generated using Macaulay2. Here $W_2$ is the RevLex segment, and $W_3$ the Lex segment. The multiplicities are $e([W_1])=13$, $e([W_2])= \ldots = e(\K[W_7])=15$, and $e(\K[W_8])=16$, so $\K[W_1]$ has the minimal Hilbert function. 
\end{ex}

It is not obvious which monomial ordering(s) that has $W_1$ in Example \ref{ex:deg3.1} as an initial segment. Another approach would be to look for a combinatorial description of the strongly stable sets that gives minimal multiplicity.

\begin{questions}
\ 
 \begin{itemize}
  \item Which monomial orderings define subalgebras with minimal Hilbert function? 
  \item Is there a combinatorial classification of the strongly stable sets giving minimal Hilbert function (not necessarily referring to monomial orderings)? 
 \end{itemize}
 
 One may also consider the questions 1 and 3 in \cite[Questions 3.6]{Boij-Conca}, mentioned in the introduction, again for $d \ge 3$. Does examples such as Example \ref{ex:persist_counterex}, where the Hilbert function is minimal in the asymptotic sense but not minimal for small arguments, exist also in higher degrees? The following example, with $d=3$, shows that a minimal value of the Hilbert function in $i=2$ does not imply minimal Hilbert function for all $i$. That is, the answer to question 3 is negative, also in degree three.  
 
 \begin{ex}
  For $n=6$, $d=3$, and $u=43$ there are 672 strongly stable sets. The sets were generated using Macaulay2. Among the algebras generated by those sets, the minimal multiplicity is $176$, and this is attained only by the set $W_1=\st(x_3x_5x_6)$. The minimal value of $\HF(A,2)$, among the 672 algebras, is 343, and is attained by both $W_1$ and $W_2=\st(x_2x_5x_6, x_4x_5^2)$. For $i>2$ we have $\HF(\K[W_1],i)<\HF(\K[W_2,i])$. 
  
  We may also remark that neither $W_1$ nor $W_2$ is a Lex or RevLex segment, as the Lex segment is  $\st(x_3x_4x_6,x_2x_6^2)$, and the RevLex segment is $\st(x_2x_4x_6,x_3^2x_6,x_5^3)$.  \end{ex}

\end{questions}

\subsection*{Acknowledgement}
First of all, I would like to thank Aldo Conca for pointing out the important connection between the multiplicity and the maximal NE-paths, and Per Alexandersson for suggesting the method of computation in Appendix \ref{app:data}. I would also like to thank Ralf Fröberg, Christian Gottlieb, and Samuel Lundqvist for our many discussions around the topics of this paper. Finally, I thank the anonymous referees for their careful reading and valuable comments.

\bibliographystyle{plain}
\bibliography{paper}
\pagebreak

 \appendix
 
 \section{Data for $n \le 80$}\label{app:data}
 In the table on the next page the monomial orderings giving the algebras on $u=\binom{n}{2}+r$ generators with minimal Hilbert function are given, for $n \le 80$ and $1 \le r \le n-3$. For $n-2 \le r \le n$ there is only one strongly stable set, as we saw in Remark \ref{rmk:large_u}. 
 
 The data for the table is based only on a computation of the multiplicities, except for three cases where more information was needed, see the discussion after Remark \ref{rmk:data}. The multiplicities are computed recursively, in the following way. We let the strongly stable sets be represented by diagrams, as before. To each box on the diagonal we also associate the number of maximal NE-paths starting in that box. The multiplicity is the sum of those numbers. Suppose that we have all diagrams, including the numbers in the diagonal boxes, of strongly stable sets with precisely $n$ columns and size greater than $\binom{n}{2}$. The following steps generate the diagrams with precisely $n+1$ columns and size greater than $\binom{n+1}{2}$. The procedure is also illustrated in Figure \ref{fig:algorithm}.
 
 \begin{enumerate}
  \item To each diagram, add one box to the left in each row. This gives the diagram a new diagonal, to which we shall associate numbers. The box in the first row is given the number 1. To the other boxes, assign the sum of the number above and the number to the right.   
  
  \item For each diagram constructed in step 1, construct a new diagram by adding a new row with one box. This box is assigned the same number as the box right above. 
  
  \item Take all diagrams produced in step 1 and 2, and discard those of size less than or equal to $\binom{n+1}{2}$.
 \end{enumerate}
 
 \begin{figure}[ht]
 \begin{minipage}{0.25\textwidth}
  \includegraphics[scale=0.7]{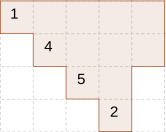}
   \end{minipage} \scalebox{1.5}{$\dashrightarrow$} \begin{minipage}{0.3\textwidth}
  \hspace{0.05\textwidth}\includegraphics[scale=0.7]{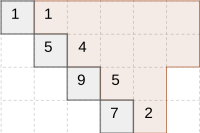}
 \end{minipage} \begin{minipage}{0.3\textwidth}
  \includegraphics[scale=0.7]{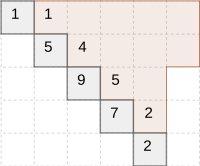}
 \end{minipage}\caption{A diagram of five columns generates two diagrams of six columns.}\label{fig:algorithm}
 \end{figure}

 The number associated to a box on the new diagonal indeed gives the number of maximal NE-paths, as each path has to start with either a step up or right. Let $L$ be an arbitrary diagram with $n+1$ columns and size greater than $\binom{n+1}{2}$. Let $L'$ be the diagram with $n$ columns obtained by removing the diagonal from $L$. Then step 1 or step 2 above applied to $L'$ will produce $L$, but we need to verify that $L'$ has size greater than $\binom{n}{2}$. If the diagonal of $L$ has at most $n$ boxes, $L'$ has size greater than $\binom{n+1}{2}-n=\binom{n}{2}$. If the diagonal of $L$ has $n+1$ boxes it means that $L$ is the largest possible diagram with $n+1$ columns, which has size $\binom{n+2}{2}$. Then $L'$ has size $\binom{n+1}{2}-(n+1)=\binom{n+1}{2}$.   
 
 Starting from the single strongly stable set $\{x_1^2\}$ on one variable we can produce all strongly stable sets on $n$ variables of size greater than $\binom{n}{2}$, for any given $n$.  
 
 To implement the algorithm, each strongly stable set can be represented by a vector containing the number of boxes in each row of the diagram.

 \includepdf[pages=-]{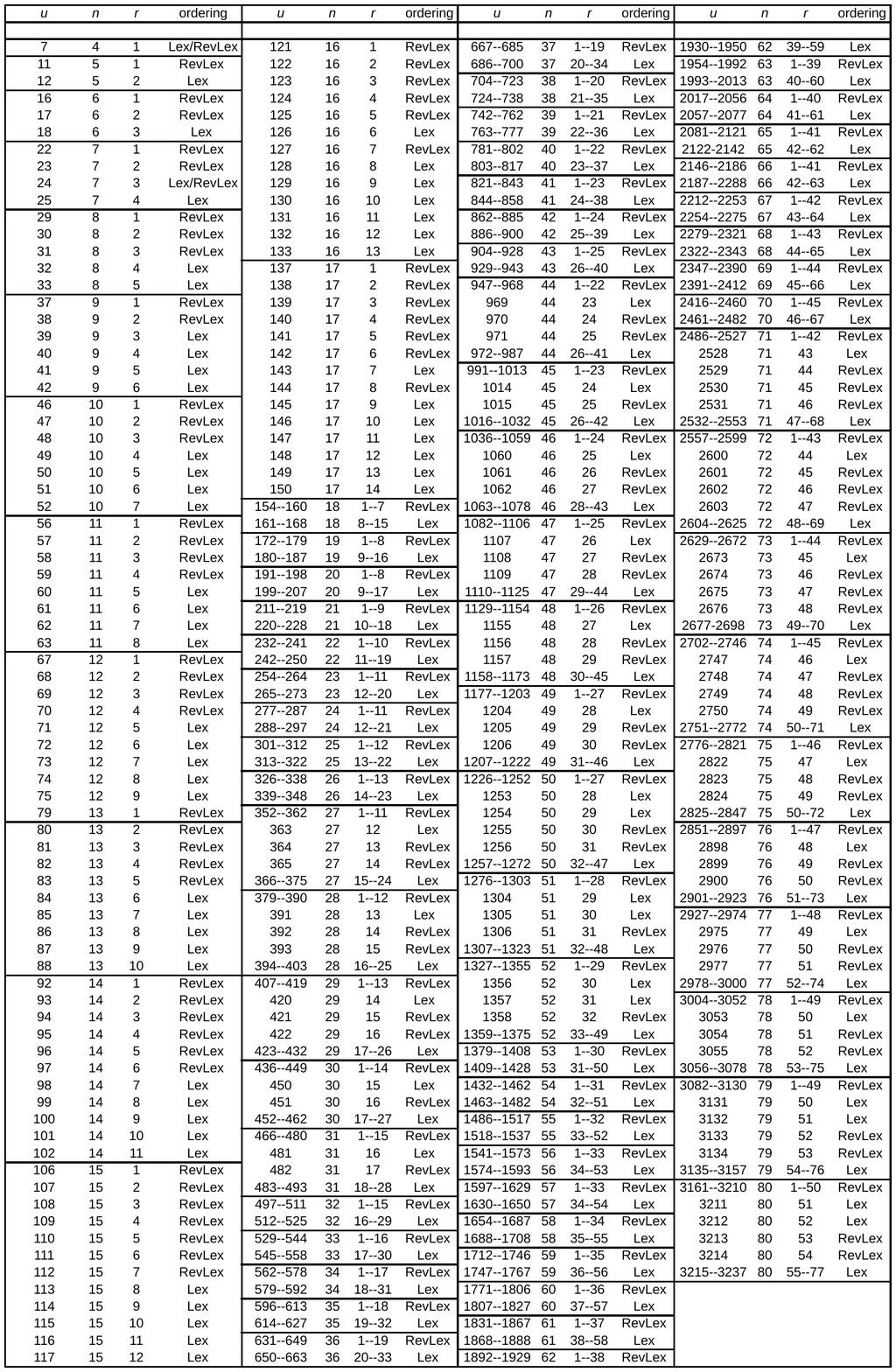}

\section{A combinatorial analogue of Conjecture \ref{conj}}\label{app:comb_conj}
Each strongly stable set $W$ of degree two monomials in $\K[x_1, \ldots,
x_n]$ corresponds to the integer partition $\lambda=(\lambda_1, \ldots,
\lambda_n)$ of $u=|W|$ defined by
\[ \lambda_i=|\{x_ix_j \in W \ | \ j\ge i\}| = \ \mbox{the length of row }
i \ \mbox{in} \ L\]
where $L$ is the diagram representing $W$. The partition $\lambda$ will have \emph{distinct parts} in the sense that $\lambda_i\ge \lambda_{i+1}$ with equality only when $\lambda_i=\lambda_{i+1}=0$. The diagram $L$ is also called the \emph{shifted Ferrers diagram} of the partition $\lambda$. We assume that $ \binom{n}{2} <u \le
\binom{n+1}{2}$ as before, meaning that the diagram has precisely $n$ columns, or equivalently that $\lambda_1=n$.
The set $\RevLex(u)$ corresponds to the partition $ (n, \ldots, \hat{i} \ldots, 1)$ meaning that we list all integers between $n$ and $1$, except $i$, where $i$ is chosen uniquely so that the parts add up to $u$. For example, the set $\RevLex(71)$ displayed in Figure \ref{fig:ex_persist1} corresponds to the partition $(12, \ldots, \hat{7}, \dots 1)=(12, 11, 10,9, 8,6,5, 4,3,2, 1)$. The set $\Lex(u)$ corresponds to the partition $(n,n-1, \ldots, j, k)$ where again $j$ and $k$ are chosen uniquely so that the sum is $u$, and with the condition that $j>k$. For example the set $\Lex(71)$ in Figure \ref{fig:ex_persist5} gives $(12,11,10,9,8,7,6,5,3)$. 

\begin{figure}[ht]
\begin{center}
\includegraphics[scale=0.4]{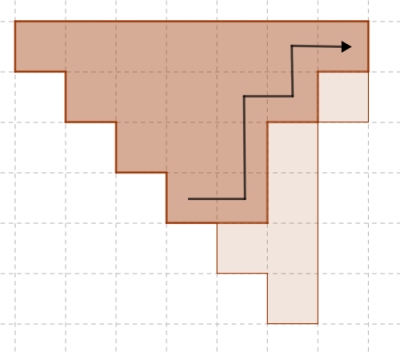}
 \caption{A maximal NE-path defines a subdiagram.}
 \label{fig:subdiagram}
\end{center}
\end{figure}

If we fix $n$, the set $W$ is uniquely determined by the partition $\lambda'=(\lambda_2, \ldots, \lambda_n)$, as $\lambda_1=n$. This is a partition of $v=u-n$, and $\binom{n-1}{2} \le v \le \binom{n}{2}$. A maximal NE-path in the diagram $L$ defines a subdiagram by taking the boxes on the path, and those to the left of it, as in Figure \ref{fig:subdiagram}. This, in turn, gives a subpartition $\mu \subseteq \lambda$ with distinct parts, i.\,e.\ $\mu=(\mu_1, \ldots, \mu_n)$ with $\mu_i \le \lambda_i$. As we will always have $\mu_1=n$, it is enough to consider $\mu'=(\mu_2, \ldots, \mu_n) \subseteq \lambda'$. In this way we have a bijection between the subpartitions $\mu'$ with distinct parts, and the maximal NE-paths of $L$. Conjecture \ref{conj} can now be stated as follows.

\begin{conjecture}\label{conj_comb}
 For fixed positive integers $N$ and $v$ such that $\binom{N}{2}\le v\le \binom{N+1}{2}$ let $\mathcal{P}$ be the set of integer partition of $v$ into distinct parts, with largest part at most $N$. The member of $\mathcal{P}$ that has the minimal number of subpartitions with distinct parts is
 \[(N, \ldots, \hat{i}, \ldots, 1)    \ \mbox{or} \  (N, N-1, \ldots, j,k) \ \mbox{with} \ j>k.\]
\end{conjecture}

Recall that the \emph{dominance order} on the set of partitions of a number $v$ is defined as follows. Let $\tau=(\tau_1, \ldots, \tau_n)$ and $\lambda=(\lambda_1, \ldots, \lambda_n)$ be partitions of $v$ with $\tau_1 \ge \tau_2 \ge \dots \ge \tau_n \ge 0$ and $\lambda_1 \ge \lambda_2 \ge \dots \ge \lambda_n \ge 0$. We say that $\tau \le \lambda$ if 
\[
 \tau_1 + \dots + \tau_k \le \lambda_1 + \dots + \lambda_k \ \mbox{for all} \ 1 \le k \le n.
\]
We can also say that $\tau \le \lambda$ if the diagram of $\tau$ can be obtained by that of $\lambda$ by ``moving boxes down to the left'' in the (shifted) Ferrers diagram. With this ordering, $\mathcal{P}$ is a bounded poset, with the two partitions in Conjecture \ref{conj_comb} as lower and upper bound.

\end{document}